\documentclass{amsart}

\usepackage{amsmath}
\usepackage{amssymb}
\usepackage{enumitem}
\usepackage{hyperref}

\newtheorem{theorem}{Theorem}[section]
\newtheorem{lemma}[theorem]{Lemma}
\newtheorem{proposition}[theorem]{Proposition}
\newtheorem{corollary}[theorem]{Corollary}

\theoremstyle{definition}

\theoremstyle{remark}

\numberwithin{equation}{section}

\newcommand{\R}{\mathbb{R}}
\newcommand{\C}{\mathbb{C}}

\newcommand{\SO}{\mathrm{SO}}
\newcommand{\SL}{\mathrm{SL}}
\newcommand{\wtSO}{\widetilde{\mathrm{SO}}}
\newcommand{\OO}{\mathcal{O}}
\newcommand{\HH}{\mathcal{H}}
\newcommand{\GG}{\mathcal{G}}
\newcommand{\VV}{\mathcal{V}}

\newcommand{\g}{\mathfrak{g}}
\newcommand{\h}{\mathfrak{h}}

\newcommand{\so}{\mathfrak{so}}
\newcommand{\spi}{\mathfrak{sp}}

\newcommand{\Ad}{\operatorname{Ad}}
\newcommand{\Iso}{\operatorname{Iso}}
\newcommand{\Kill}{\operatorname{Kill}}
\newcommand{\Aut}{\operatorname{Aut}}

\begin{document}

\title[Local and global rigidity for isometric actions]{Local and global rigidity for isometric actions of simple Lie groups on pseudo-Riemannian manifolds}

\author{Raul Quiroga-Barranco}
\address{Centro de Investigaci\'on en Matem\'aticas, Jalisco S/N, Col. Valenciana, Guanajuato, Guanajuato, 36023, Mexico}
\email{quiroga@cimat.mx}

\dedicatory{To Jimmie D. Lawson on the occasion of his 50 years as researcher at LSU.}

\thanks{This research was supported by a Conacyt grant and by SNI}


\maketitle

\begin{abstract}
    Let $M$ be a finite volume analytic pseudo-Riemannian manifold that admits an isometric $G$-action with a dense orbit, where $G$ is a connected non-compact simple Lie group. For low-dimensional $M$, i.e.~$\dim(M) < 2\dim(G)$, when the normal bundle to the $G$-orbits is non-integrable and for suitable conditions, we prove that $M$ has a $G$-invariant metric which is locally isometric to a Lie group with a bi-invariant metric (local rigidity theorem). The latter does not require $M$ to be complete as in previous works. We also prove a general result showing that $M$ is, up to a finite covering, of the form $H/\Gamma$ ($\Gamma$ a lattice in the group $H$) when we assume that $M$ is complete (global rigidity theorem). For both the local and the global rigidity theorems we provide cases that imply the rigidity of $G$-actions for $G$ given by $\SO_0(p,q)$, $G_{2(2)}$ or a non-compact simple Lie group of type $F_4$ over $\R$. We also survey the techniques and results related to this work.
\end{abstract}

\section{Introduction}\label{sec:intro}
In this paper we consider the dynamical systems associated to a non-compact simple Lie group $G$. There are plenty of examples of interesting $G$-actions that can be obtained through an algebraic construction. A most notable one is given by choosing a non-trivial homomorphism $G \rightarrow H$ into a non-compact type semisimple Lie group $H$ and a lattice $\Gamma \subset H$. These yield a $G$-action on the finite volume manifold $H/\Gamma$. Such examples are very important because of their complicated dynamics: they are ergodic for irreducible lattices. Furthermore, these are also interesting from a geometric viewpoint, since the $G$-action on $H/\Gamma$ is isometric for the pseudo-Riemannian metric coming from the bi-invariant one on $H$ defined by the Killing form of its Lie algebra. Hence, a natural problem is to study the properties of isometric $G$-actions on pseudo-Riemannian manifolds.

On the other hand, it has been proved over the last decades that $G$-actions with complicated dynamics have very rigid properties, even when they are only finite volume preserving. Nevertheless, some of the most interesting results have been obtained under the existence of some sort of geometric invariant. In particular, such results can be found in many of the works in our bibliography. Among these works, two have had a strong impact: Gromov's Rigid Transformation Groups \cite{Gromov} and Zimmer's paper on Automorphism groups of geometric manifolds \cite{ZimmerAutGeom}. These developed tools that have been used to exhibit the rigid behavior of $G$-actions in the presence of an invariant geometric structure.

In this work, we build on the theory from \cite{Gromov}, \cite{ZimmerAutGeom} as well as our own work to obtain results that prove the rigid behavior of $G$-actions preserving a pseudo-Riemannian metric. Our main setup is that of a finite volume analytic pseudo-Riemannian manifold $M$ carrying an isometric $G$-action with a dense orbit. In the past, we have shown that under the assumption of completeness of the metric, the manifold $M$ can be built from other Lie groups (see \cite{QuirogaAnnals}, \cite{QuirogaCh}, \cite{OQ-SO}, \cite{OQ-U}, \cite{QuirogaG2}, \cite{QuirogaRoblero}). Similar results have been obtained by other authors, see for example \cite{Sedano}. In all these works, it has been proved for several cases that if $M$ is low-dimensional with respect to $G$, i.e.~$\dim(M) < 2\dim(G)$, then $M$ is (up to a finite cover) $G$-equivariantly equivalent to either $(N \times G)/\Gamma$, for some discrete subgroup $\Gamma$ of $\Iso(N)\times G$, or to $H/\Gamma$ as above.

Within the framework just described, the first goal of this work is to present a local rigidity result, Theorem~\ref{thm:local-rigidity-general}, that deals with $G$-actions where $M$ is not necessarily complete. We consider the low-dimensional case $\dim(M) < 2\dim(G)$ and assume that the normal bundle to the $G$-orbits is non-integrable, as well as some other technical conditions, to conclude that the manifold $M$ has a $G$-invariant metric which is locally isometric to a simple group with a bi-invariant metric. Theorem~\ref{thm:local-rigidity-general} is very general but, most importantly, we provide several cases for which its hypotheses are satisfied. In particular, we obtain Corollaries~\ref{cor:so-local-rigidity}, \ref{cor:g22-local-rigidity} and \ref{cor:f4-local-rigidity} for which we conclude that the manifold $M$ has a $G$-invariant metric that is locally isometric to a bi-invariant metric on a Lie group for actions of $\SO_0(p,q)$, $G_{2(2)}$ and the non-compact simple Lie groups of type $F_4$ over $\R$, respectively.

The second goal is to put into perspective some of the previous results that have been proved before under the assumption of completeness. In this context, the main result is Theorem~\ref{thm:global-rigidity-general} which adds to the hypotheses of Theorem~\ref{thm:local-rigidity-general} the requirement of completeness to obtain a global description, up to a finite cover, of the type $H/\Gamma$ for the manifold $M$. Again, Theorem~\ref{thm:global-rigidity-general} is very general and requires some technical conditions, however it is possible to obtain this in several cases from more elementary properties. In this way we prove Corollaries~\ref{cor:so-global-rigidity}, \ref{cor:g22-global-rigidity} and \ref{cor:f4-global-rigidity} that establish global rigidity results for actions of $\SO_0(p,q)$, $G_{2(2)}$ and the non-compact simple Lie groups of type $F_4$ over $\R$, respectively. The first two of these corollaries have already appeared in \cite{OQ-SO} and \cite{QuirogaG2}, respectively, but the third one is new to the best of our knowledge.

Finally, we also present in the first few sections a short panoramic survey of the results related to the main theorems.

\section{Local freeness of isometric actions}\label{sec:localfree}
Let $G$ be a connected non-compact simple Lie group and let $M$ be a pseudo-Riemannian manifold with finite volume. We will assume in the rest of this work that $G$ acts non-trivially and smoothly on $M$ preserving the metric, but we will add further conditions as needed.

It was first proved in \cite{ZimmerSSAut} that for such actions the stabilizers have only the trivial possibilities. More precisely, there is the following general result for $G$-actions.

\begin{proposition}[Zimmer \cite{ZimmerSSAut}]
    For every finite measure preserving ergodic $G$-action the stabilizers are either $G$ or discrete on a conull set.
\end{proposition}

At the time, this was already interesting and very useful. Nevertheless, it was proved about one decade latter that a much stronger conclusion holds for this sort of actions.

\begin{proposition}[Zimmer \cite{StuckZimmer}]
    Assume that $G$ has finite center and real rank at least $2$. Then, any faithful, ergodic and finite measure preserving $G$-action is essentially free: i.e.~the stabilizer of almost every point is trivial.
\end{proposition}

As noted in the introduction, the development of rigidity results for $G$-actions built many important conclusions from the assumption on the existence of some invariant geometric structure. Note that this where smoothness actually plays a role and, in some cases, there is a chance to obtain conclusions at every single point. In his celebrated ``Rigid transformation groups'', Gromov established one of the first everywhere local freeness result.

\begin{proposition}[Gromov \cite{Gromov}5.4.A]
	If $M$ is Lorentz, then the $G$-action on $M$ is locally free.
\end{proposition}

It was this sort of results that started the study of $G$-actions on Lorentzian manifolds. In particular, these lead to Zimmer's work on the group of automorphisms of a Lorentzian manifold. More precisely, it was proved in \cite{ZimmerLorentz} that, in the non-amenable case, the connected component of the automorphism group of a compact Lorentzian manifold is locally isomorphic to $\SL(2,\R)\times K$ for some compact group $K$.

However, it turns out that local freeness for isometric $G$-actions holds for a more general setup.

\begin{proposition}[Szaro \cite{Szaro}]
    For an isometric $G$-action on $M$ as above, local freeness holds everywhere in the following cases.
	\begin{enumerate}
	   \item The $G$-action has a dense orbit.
	   \item The pseudo-Riemannian manifold $M$ is compact and complete.
	\end{enumerate}
\end{proposition}

About at the same time, the following results were obtained where local freeness can also be concluded.

\begin{proposition}[Zeghib \cite{ZeghibAffine}]
	For an isometric $G$-action on $M$ as above, local freeness holds everywhere in the following cases.
	\begin{enumerate}
        \item All the $G$-orbits are isotropic.
        \item The $G$-action is ergodic.
        \item The group $G$ is split and not locally isomorphic to $\SL(2,\R)$.
	\end{enumerate}
\end{proposition}

From now on, we will assume that the $G$-action on $M$ has a dense orbit and so it is locally free as well. In particular, the set of $G$-orbits in $M$ defines a smooth foliation $\OO$, whose tangent bundle can be trivialized by the $G$-action. This is seen by considering the map
\begin{align*}
	\Psi: M \times \g &\rightarrow T\OO \\
	(x,X) &\mapsto X^*_x.
\end{align*}
Here we denote by $X^*$ the vector field on $M$ whose flow is the $1$-parameter group $(\exp(tX))_{t\in \R}$ for $X \in \g$. A similar notation will be used whenever an action a of group is involved. The map $\Psi$ is an isomorphism of vector bundles over $M$. This map allows us to pull-back the metric on the fibers of $T\OO$ (inherited from $M$) to symmetric bilinear forms over $\g$. In other words, we have the map
\begin{align*}
	\Phi : M &\rightarrow \mathrm{Symm}^2(\g) \\
		\Phi(x)(X,Y) &= \left<X^*_x, Y^*_x\right>_x,
\end{align*}
for every $x \in M$ and $X,Y \in \g$, where $\left<\cdot,\cdot\right>$ denotes the metric of $M$. We have also denoted by $\mathrm{Symm}^2(\g)$ the space of symmetric bilinear forms on $\g$.

A straightforward computation shows that $\Phi$ is $G$-equivariant. This uses that the $G$-action is isometric and the easy to prove identity
\[
	d g_x (X^*_x) =  (\Ad(g)(X)^*)_{gx},
\]
for all $x \in M$, $g \in G$ and $X \in \g$.

The following is a less trivial fact (see \cite{Szaro}, \cite{ZeghibAffine} and \cite{QuirogaAnnals}).

\begin{proposition}
    The map $\Phi$ is $G$-invariant. In particular, if the $G$-action is ergodic or has a dense orbit, then $\Phi$ is a constant function whose value is a fixed bi-invariant bilinear form on $\g$.
\end{proposition}

Since $G$ is simple, any bi-invariant bilinear form is either $0$ or non-degenerate. Hence, we obtain the following result.

\begin{corollary}\label{cor:metric_orbits}
    If the $G$-action on $M$ is either ergodic or has a dense orbit, then one of the following holds.
	\begin{enumerate}
        \item The $G$-orbits, and so $T\OO$, are isotropic.
        \item There is a bi-invariant pseudo-Riemannian metric on $G$ so that for every $x \in M$, the orbit map
			\begin{align*}
				G &\rightarrow Gx \\
				g &\mapsto gx,
            \end{align*}
            is a local isometry where $Gx$ carries the inherited metric of $M$.
	\end{enumerate}
\end{corollary}

Once this result has been achieved, we have complete knowledge of the geometry of $M$ along the $G$-orbits, except when these are isotropic. Recall that the bi-invariant metrics on $G$ are all easily given in terms of the Killing form of $\g$ and a complex structure on the latter when such exists.

By the previous remarks, the following result turns out be very useful. For the proof we refer to \cite{QuirogaCh}

\begin{proposition}\label{prop:dimM<2dimG}
    If the isometric $G$-action on $M$ has a dense orbit and $\dim(M) < 2 \dim(G)$, then the $G$-orbits are non-degenerate pseudo-Riemannian manifolds of $M$.
\end{proposition}

It remains to find a way to study the ``transverse'' to the foliation $\OO$ by $G$-orbits.

\section{Killing fields and Gromov-Zimmer machinery}\label{sec:KillingGromovZimmer}
In the first place, the most natural place to start is by considering Killing fields. We recall that a Killing field for a pseudo-Riemannian manifold is a smooth vector field whose one-parameter group of local diffeomorphisms consists of local isometries. For a pseudo-Riemannian manifold $N$ we will denote by $\Kill(N)$ the space of Killing fields globally defined on $N$. It is well known that $\Kill(N)$ is a Lie algebra. Furthermore, $\Kill(N)$ is a finite dimensional Lie algebra. We observe that a Killing field may not integrate to globally defined isometries, since the field may not be complete. However, it is known that the completeness of a pseudo-Riemannian manifold implies the completeness of every Killing field on it (see \cite{ONeill}).

On the other hand, on a pseudo-Riemannian manifold $N$ we can consider locally defined Killing fields that we will denote by $\Kill^{loc}(N)$. Again, it may happen that a locally defined Killing field does not extend to a global one. It was proved by Nomizu \cite{Nomizu} that every local Killing field (with an open connected domain) can be extended to a global one as long as a regularity condition is satisfied for the manifold. We refer to \cite{Nomizu} for the precise statement. It was also proved that such condition holds for analytic manifolds. We also observe that the results in \cite{ONeill} were stated for Riemannian manifolds, but those that are used here also hold for the pseudo-Riemannian case with the same proof.

Beyond these classical facts, Gromov \cite{Gromov} introduced a collection of tools that allows us to conclude the existence of further symmetries under our current conditions. The main point was to build infinitesimal Killing fields that are then extended to local and then global ones. For the following discussion we refer to \cite{GCT} and \cite{QuirogaCh}, particularly for the notation on jets of functions. For a pseudo-Riemannian manifold $N$, we will denote by $\Kill^k(N,x)$ the space of $k$-jets at $x$ of vector fields that preserve the metric, up to order $k$. The subspace of such infinitesimal Killing vector fields that vanish at $x$ is denoted by $\Kill^k_0(N,x)$. Similarly, we denote by $\Aut^k(N,x)$ the group of $k$-jets at $x$ of diffeomorphisms that fix $x$ and preserve the metric up to order $k$. Following this notation, we will write $\Kill_0(N,x)$ and $\Kill^{loc}_0(N,x)$ for the global and local Killing fields of $N$ that vanish at the point $x$.

It follows from \cite{Nomizu} that infinitesimal Killing fields can be extended to local and the global ones if the order is high enough and regularity conditions hold. We state here the corresponding result for analytic manifolds, for which the required regularity holds everywhere. We refer to \cite{GCT} and \cite{QuirogaCh} for further details.

\begin{proposition}
    Let $N$ be an analytic pseudo-Riemannian manifold. Then, for every $x \in N$ there exists a positive integer $k(x)$ such that the map
	\begin{align*}
		\Kill_0(N,x) &\rightarrow \Kill^k_0(N,x) \\
			X &\mapsto j^k(X),
	\end{align*}
    is an isomorphism for every $k \geq k(x)$.
\end{proposition}

The conclusion is that to obtain global symmetries in the form of Killing fields, it is enough to build infinitesimal ones of high enough order. This is the setup corresponding to fiber bundles over a manifold. In the first place, it is possible to embed the isometries of a pseudo-Riemannian manifold $N$ into its linear frame bundle $L(N)$ by the map
\begin{align*}
	\Iso(N) &\rightarrow L(N) \\
	\varphi &\mapsto j^1_{x_0}(\varphi)
\end{align*}
where we have fixed a point $x_0 \in N$ as a reference. This is in fact the main step used to prove that $\Iso(N)$ is a Lie group (see \cite{KobayashiTransGroups}). This can be generalized to a map defined on $\Aut^k(N,x)$ into higher order frame bundles. Then, the following result from \cite{ZimmerLorentz} becomes very useful.

\begin{proposition}[Zimmer \cite{ZimmerLorentz}]
    Let $G$ be as above, and assume that it acts on a principal bundle $P \rightarrow N$ with algebraic structure group $H$ and preserving an ergodic finite measure on $N$. Then, there is an embedding of Lie algebras $\g \subset \h$.
\end{proposition}

Once these tools have been developed, we can obtain the following result. We refer to \cite{QuirogaCh} for the proof and further details.

\begin{theorem}\label{thm:rho-map}
    Suppose that $G$ and $M$ are as before with $M$ analytic, that the $G$-action on $M$ is analytic, isometric and has a dense orbit. Then, there exists a dense conull subset $S \subset \widetilde{M}$ such that for every $x \in S$ the following properties are satisfied.
	\begin{enumerate}
        \item There is an injective homomorphism of Lie algebras $\rho_x : \g \rightarrow \Kill(\widetilde{M})$, which is then an isomorphism onto its image $\g(x) = \rho_x(\g)$.
        \item Every element of $\g(x)$ vanishes at $x$, in other words, we have $\g(x) \subset \Kill_0(\widetilde{M},x)$.
        \item For every $X,Y \in \g$ we have
		\[
			[\rho_x(X),Y^*] = [X,Y]^*,
		\]
		for all $X,Y \in \g$.
	\end{enumerate}
\end{theorem}

The previous result has a number of consequences that we now describe. First, we recall that the $G$-action on $M$ lifts to a $\widetilde{G}$-action on $\widetilde{M}$. With this setup, we will denote by the same symbol $\OO$ the foliation by $\widetilde{G}$-orbits on $\widetilde{M}$.
	
The subbundle $T\OO$ is clearly pointwise generated by the vector fields $Y^*$ with $Y \in \g$. Hence, it follows that the vector fields belonging to the Lie algebras $\g(x)$ preserve $T\OO$. Since such vector fields are also Killing fields, it follows that the elements of $\g(x)$ preserve $T\OO^\perp$ and $\OO$ as well.
	
It is well known that, in every manifold $N$, the vector fields that vanish at the point $x$ admit a representation on the tangent space at $x$. For such a vector field $Z$ the corresponding linear map is given by
\begin{align*}
	\lambda_x(Z) : T_xN &\rightarrow T_xN \\
	v &\mapsto [Z,V]_x
\end{align*}
where $V$ is any (local) vector field such that $V_x = v$. In particular, for every pseudo-Riemannian manifold $N$, we have a representation of $\Kill_0(N,x)$ given by
\begin{align*}
	\Kill_0(N,x) &\rightarrow \so(T_xN) \\
		Z &\mapsto \lambda_x(Z).
\end{align*}
Note that $\lambda_x(Z)$ lies in $\so(T_xN)$ because $Z$ is a Killing field.
	
The previous discussion and Theorem~\ref{thm:rho-map} imply that we have a representation of $\g$ given by
\[
	\lambda_x \circ \rho_x : \g \rightarrow \so(T_x\widetilde{M})
\]
for which the $T_x\OO$ and $T_x\OO^\perp$ are $\g$-submodules.
	
Theorem~\ref{thm:rho-map} already provides a description of the representation of $\g$ on $T_x\widetilde{M}$. We recall that for every $x \in \widetilde{M}$ the map
\[
	X \mapsto X^*_x
\]
yields a natural identification $T_x\OO \simeq \g$. Hence, Theorem~\ref{thm:rho-map}(3) says that, with respect to this identification, $T_x \OO$ is isomorphic as a $\g$-module to the adjoint representation of $\g$.

At this point we recall the problem posed at the end of Section~\ref{sec:localfree}: the study of the transverse direction to the foliation $\OO$ by $G$-orbits. Using the setup we have described, we need to consider the representation of $\g$ on the vector spaces $T_x \OO^\perp$. However, for this to be an actual transverse to the foliation $\OO$ we will consider the case $\dim(M) < 2 \dim(G)$. As noted in Proposition~\ref{prop:dimM<2dimG} this implies that the foliation $\OO$ has leaves that are non-degenerate for the metric on $M$. In the rest of this work we will assume that such dimension restriction holds.

Hence, we have an orthogonal non-degenerate decomposition
\[
	T\widetilde{M} = T\OO \oplus T\OO^\perp,
\]
and using the identification $\Psi : M \times \g \rightarrow T\OO$ from Section~\ref{sec:localfree} we realize the orthogonal projection $T \widetilde{M} \rightarrow T\OO$ as a map
\begin{align*}
	\omega : T\widetilde{M} &\rightarrow \g \\
		v &\mapsto X,
\end{align*}
where $v = X^*_x + u$, for some $X \in \g$ and $u \in T_x\OO^\perp$. We also consider the restriction $\Omega$ of $d\omega$ to $\wedge^2 T\OO^\perp$. It turns out that both maps are related to the $\g$-module structures described above with $\Omega$ providing an invariant that measures the integrability of $T\OO^\perp$. For the proof of the next result we refer to \cite{OQ-SO}.

\begin{proposition}\label{prop:o_O_homomorphisms}
    Let $G$, $M$ and $S$ be as in Theorem~\ref{thm:rho-map}, and assume that the $G$-orbits are non-degenerate submanifolds of the pseudo-Riemannian manifold $M$. Then, we have the following.
	\begin{enumerate}
        \item For every $x \in S$, the maps $\omega_x : T_x \widetilde{M} \rightarrow \g$ and $\Omega_x : \wedge^2 T_x \OO^\perp \rightarrow \g$ are homomorphisms of $\g$-modules for the $\g$-module structures described above.
        \item The normal bundle $T\OO^\perp$ is integrable if and only if $\Omega = 0$.
	\end{enumerate}
\end{proposition}

The first case to consider is that when the bundle $T\OO^\perp$ is integrable. When this occurs, one can describe the manifold $M$ as a local pseudo-Riemannian product. The most general result is the following. This behavior was first observed by Gromov \cite{Gromov} for Lorentz manifolds. We refer to \cite{QuirogaCh} for the proof.

\begin{theorem}\label{thm:GxNcase}
    Suppose that the isometric $G$-action on $M$ has a dense orbit and that the metric of $M$ has finite volume and is complete. If the $G$-orbits are non-degenerate and the normal bundle $T\OO^\perp$ is integrable, then there exist
	\begin{enumerate}
        \item an isometric finite covering $\widehat{M} \rightarrow M$ to which the $G$-action lifts,
        \item a simply connected pseudo-Riemannian manifold $\widetilde{N}$, and
        \item a discrete subgroup $\Gamma \subset G \times \Iso(\widetilde{N})$
	\end{enumerate}
    such that $\widehat{M}$ is $G$-equivariantly isometric to $(G\times\widetilde{N})/\Gamma$.
\end{theorem}

It is worthwhile to know that for suitable transverse structures for the $G$-orbits of these manifolds, it is possible to provide a more detailed description of $M$. We refer to \cite{QuirogaAnnals} and \cite{QuirogaCh} for such results.

\section{Structure of the Lie algebra centralizer of the action}\label{sec:centralizer}
In order to have a global description of the isometric $G$-action on $M$ as discussed in the previous sections, it is useful to consider the Killing centralizer for the $G$-action. More precisely, we now want to consider the Lie algebra
\[
	\HH = \{ X \in \Kill(\widetilde{M}) \mid dg(X) = X
			\text{ for all } g \in \widetilde{G} \}.
\]

In the first place, Gromov \cite{Gromov} proved that this centralizer is transitive on an open dense set under our current conditions. For a proof of the next result we refer to \cite{ZimmerEntropy}.

\begin{proposition}[Gromov \cite{Gromov}]\label{prop:HH-transitivity}
    Suppose that the isometric $G$-action on $M$ has a dense orbit and that $M$ and the $G$-action are analytic. Then, there is an open dense conull subset $U \subset \widetilde{M}$ such that $ev_x(\HH) = T_x \widetilde{M}$ for every $x \in U$.
\end{proposition}

The next step is to relate this centralizer $\HH$ to the $\g$-module structures considered in Section~\ref{sec:localfree}. This is achieved by considering the map introduced in the following result.

\begin{proposition}\label{prop:rhox-GG(x)}
    Let $G$, $M$ and $S$ be as in Theorem~\ref{thm:rho-map}. Then, for every $x \in S$ and $\rho_x$ given by Theorem~\ref{thm:rho-map}, the map defined by
	\begin{align*}
		\widehat{\rho}_x : \g &\rightarrow \Kill(\widetilde{M}) \\
			\widehat{\rho}_x(X) &= \rho_x(X) + X^*,
	\end{align*}
    is an injective homomorphism of Lie algebras whose image $\GG(x)$ lies in $\HH$. In particular, it induces a $\g$-module structure on $\HH$ with $\GG(x)$ a Lie subalgebra isomorphic to $\g$ both as a Lie algebra and a $\g$-module.
\end{proposition}

It is natural and more useful to consider, for a fixed point $x \in S$, the Lie subalgebra $\GG(x) \subset \HH$ and replace the module structures over $\g$ with corresponding ones over $\GG(x)$. The first result in this direction is the following. We refer to \cite{OQ-U} for its proof.

\begin{proposition}\label{prop:GG(x)andHH0(x)}
    With the hypotheses and notation of Proposition~\ref{prop:rhox-GG(x)}, let $U$ be an open set as in Proposition~\ref{prop:HH-transitivity}. Then, for every $x \in S\cap U$, the following properties hold.
	\begin{enumerate}
		\item The map given by
			\begin{align*}
				\lambda_x : \GG(x)
                &\rightarrow \so(T_x \widetilde{M}) \\
					\lambda_x(Z)(v) &= [Z,V]_x
			\end{align*}
            where $V \in \HH$ so that $V_x = v$, is a well defined representation of $\GG(x)$.
        \item The evaluation map $ev_x : \HH \rightarrow T_x \widetilde{M}$ is homomorphism of $\GG(x)$-modules that maps $\GG(x)$ isomorphically onto $T_x\OO$.
        \item The subspace $T_x \OO^\perp$ is $\GG(x)$-submodule of $T_x\widetilde{M}$.
	\end{enumerate}
    In particular, the subspace defined by $\HH_0(x) = \ker(ev_x)$ is both a Lie subalgebra and a $\GG(x)$-submodule of $\HH$. Furthermore, the sum $\GG(x) + \HH_0(x)$ is direct and a Lie subalgebra of $\HH$ in which $\HH_0(x)$ is an ideal. Hence, $\HH$ is a module over $\GG(x) + \HH_0(x)$ as well.
\end{proposition}

Note that the last result can be applied to Proposition~\ref{prop:o_O_homomorphisms} so that we can replace the $\g$-module structure with a $\GG(x)$-module structure and obtain the same conclusion. We will use this in the rest of this work.

Through the construction of Proposition~\ref{prop:GG(x)andHH0(x)} we have enlarged the module structure of $\HH$ from $\GG(x) \simeq \g$ to one over $\GG(x) + \HH_0(x)$. One can think of $\HH_0(x)$ as the isotropy for the vector field ``action'' of $\HH$ over $\widetilde{M}$. Hence, the introduction of this module structure paves the way to understand the relationship between $\HH$ and $\widetilde{M}$. In fact, one can show that $T_x \widetilde{M}$ is a module over $\GG(x) + \HH_0(x)$ and that $ev_x$ is a homomorphism of modules over $\GG(x) + \HH_0(x)$. We let the reader see the details in \cite{OQ-U}.

An important piece of information is to figure out the size of the isotropy $\HH_0(x)$. The following result provides a key simplification and shows that we only have to look at the representation on the normal bundle. A proof can be found in \cite{OQ-U}.

\begin{proposition}\label{prop:lambdaperp_on_HH0(x)}
    With the hypotheses and notation of Proposition~\ref{prop:rhox-GG(x)}, let $U$ be an open set as in Proposition~\ref{prop:HH-transitivity}. Let us also assume that the $G$-orbits are non-degenerate submanifolds of the pseudo-Riemannian manifold $M$. Let us denote by $\lambda^\perp_x : \HH_0(x) \rightarrow \so(T_x \OO^\perp)$ the representation of $\lambda_x$ on $\HH_0(x)$ on $T_x\OO^\perp$. Then $\lambda^\perp_x$ is injective.
\end{proposition}

Next we obtain a decomposition of $\HH$ that relates to the decomposition $T_x \widetilde{M} = T_x \OO\oplus T_x \OO^\perp$ in the case of non-degenerate $G$-orbits. For the proof we refer to \cite{OQ-U}.

\begin{proposition}\label{prop:VV(x)}
    With the hypotheses and notation of Proposition~\ref{prop:rhox-GG(x)}, let $U$ be an open set as in Proposition~\ref{prop:HH-transitivity}. Let us also assume that the $G$-orbits are non-degenerate submanifolds of the pseudo-Riemannian manifold $M$. Then, for every $x \in S\cap U$ there is a $\GG(x)$-submodule $\VV(x) \subset \HH$ such that
	\[
		\HH = \GG(x) \oplus \HH_0(x) \oplus \VV(x), \quad
		T_x \OO^\perp = ev_x(\VV(x)).
	\]
\end{proposition}

To exploit this constructions we provide additional equivariance properties for the forms $\omega$ and $\Omega$. For this we first recall that there is a canonical isomorphism
\[
	\wedge^2 E \simeq \so(E)
\]
for any $E$ with a non-degenerate (pseudo-)inner product. Such isomorphism holds as vector spaces and as $\so(E)$-modules. In particular, for every $x \in S$, as in Theorem~\ref{thm:rho-map}, we can consider $\Omega_x : \so(T_x\OO^\perp) \rightarrow \g$. Then, the following intertwining properties hold for $\omega_x$ and $\Omega_x$. See \cite{OQ-U} for the proof.

\begin{proposition}\label{prop:o_O_more_intertwining}
    With the hypotheses and notation of Proposition~\ref{prop:rhox-GG(x)}, let $U$ be an open set as in Proposition~\ref{prop:HH-transitivity}. Let us also assume the $G$-orbits are non-degenerate submanifolds of $M$. Then, for every $x \in S\cap U$ the following hold.
	\begin{enumerate}
        \item For every $X \in \g$ and $Y$ smooth vector field on $\widetilde{M}$ we have
		  \[
		  	\omega_x([\rho_x(X), Y]) = [X, \omega_x(Y)].
		  \]
        \item The linear map $\Omega_x : \so(T_x \OO^\perp) \rightarrow \g$ is $\HH_0(x)$-invariant through $\lambda^\perp_x$. In other words, we have
		  \[
		  	\Omega_x([\lambda^\perp_x(Z),T]) = 0
		  \]
            for every $Z \in \HH_0(x)$ and $T \in \so(T_x \OO^\perp)$. In particular, we have
		  \[
			[\lambda^\perp_x(\HH_0(x)), \so(T_x\OO^\perp)]
                \subset \ker(\Omega_x).
		  \]
	\end{enumerate}
\end{proposition}

\section{The case of $T_x\OO^\perp$ a $\GG(x)$-irreducible module}
In this section we will assume that the for some $x \in S\cap U$, as given in Propositions~\ref{prop:HH-transitivity} and \ref{prop:rhox-GG(x)}, the $\GG(x)$-module structure of $T_x\OO^\perp$ described in Proppsition~\ref{prop:GG(x)andHH0(x)} is non-trivial and irreducible. The non-triviality can be obtained from Proposition~\ref{prop:o_O_homomorphisms} by assuming that $M$ is not a local pseudo-Riemannian product. On the other hand, as we will see in the applications, the irreducibility can be obtained by requiring suitable upper bounds on $\dim(M)$.

\begin{proposition}\label{prop:V0-inTOOperp-dim4}
    With the hypotheses and notation of Proposition~\ref{prop:rhox-GG(x)}, let $x \in S\cap U$ be such that $T_x \OO^\perp$ is a non-trivial irreducible $\GG(x)$-submodule for the structure given in Proposition~\ref{prop:GG(x)andHH0(x)}. Let $\so(T_x\OO^\perp)$ be endowed with the induced $\GG(x)$-module structure. If $V_0$ is the sum of trivial $\GG(x)$-submodules of $\so(T_x\OO^\perp)$, then $\dim(V_0) \leq 4$.
\end{proposition}
\begin{proof}
    Recall that the $\GG(x)$-module structure on $\so(T_x\OO^\perp)$ is given by the representation $\lambda_x^\perp : \GG(x) \rightarrow \so(T_x\OO^\perp)$ obtained by restricting $\lambda_x$ from Proposition~\ref{prop:GG(x)andHH0(x)} to the $\GG(x)$-submodule $T_x\OO^\perp$.
	
    If $A \in V_0$ is given, then $[\lambda_x^\perp(X),A] = 0$ for all $X \in \GG(x)$. This implies that any such $A$ is a homomorphism of $\GG(x)$-modules for $T_x\OO^\perp$. It is well known (see \cite{Oni}) that, after complexifying, the complex $\GG(x)$-module $(T_x\OO^\perp)^\C$ is either irreducible or the sum of an irreducible module and its conjugate. By Schur's lemma, it follows that the space of homomorphisms of the $\GG(x)$-module $T_x\OO^\perp$ is at most $4$-dimensional. This implies the claim.
\end{proof}

The previous result establishes a bound that together with a low dimension condition on $M$ will let us describe the structure of the centralizer $\HH$ with its Lie subalgebra $\HH_0(x)$.

\section{Local rigidity for analytic manifolds}
In the rest of this section we will assume that the manifold $M$ as well as the $G$-action on it is analytic. In view of Theorem~\ref{thm:GxNcase} we will restrict our attention to the case where the normal bundle $T\OO^\perp$ is non-integrable. In particular, Proposition~\ref{prop:o_O_homomorphisms} implies that the set of points $x \in \widetilde{M}$ for which $\Omega_x = 0$ has measure $0$. This uses the fact that both $M$ and the $G$-action on $M$ are analytic.

Before stating the main result of this section, we will discuss three cases that will satisfy its conditions. The first two were studied in \cite{OQ-SO} and \cite{QuirogaG2}. The third one yields a new result.

\subsection{The group $G = \SO_0(p,q)$ and $\dim(M) \leq (p+q)(p+q+1)/2$}
Let $p,q \geq 1$ be integers such that $p+q \not= 2,4$. We assume that $G = \SO_0(p,q)$ and also that $\dim(M) \leq (p+q)(p+q+1)/2$. Note that these restrictions imply that $\dim(M) < 2 \dim(G)$ and so Proposition~\ref{prop:dimM<2dimG} can be applied to conclude that $TM = T\OO \oplus T\OO^\perp$.

We are assuming that the normal bundle is non-integrable, and so it follows from Proposition~\ref{prop:o_O_homomorphisms} that for almost every point $x \in S\cap U \subset \widetilde{M}$ the module homomorphism $\Omega_x \not= 0$. Furthermore, Proposition~\ref{prop:HH-transitivity} implies that for almost every such $x$ we also have $ev_x(\HH) = T_x \widetilde{M}$. Let us fix any such $x$.

Since $\Omega_x : \so(T_x\OO^\perp) \rightarrow \so(p,q)$ is a non-trivial, and hence surjective, homomorphism of $\GG(x)$-modules (see Proposition~\ref{prop:GG(x)andHH0(x)} and the remarks that follow) we conclude that $T_x\OO^\perp$ is a non-trivial $\GG(x)$-module. In our case, $\GG(x) \simeq \so(p,q)$ and the dimension restriction on $M$ implies that $\dim(T_x\OO^\perp) \leq p+q$. We observe that $\R^{p,q}$ is the lowest non-trivial irreducible representation of $\so(p,q)$. Hence, it follows that $T_x\OO^\perp \simeq \R^{p,q}$ as $\GG(x)$-modules. We recall that this module admits a unique (up to a constant) $\GG(x)$-invariant non-degenerate symmetric bilinear form.

From the previous discussion we conclude that $\so(T_x\OO^\perp)$ is isomorphic to $\so(p,q)$ as $\GG(x)$-modules. Hence, $\Omega_x$ is in fact an isomorphism.

On the other hand, we recall that by Proposition~\ref{prop:lambdaperp_on_HH0(x)} we have a monomorphism $\lambda_x^\perp : \HH_0(x) \rightarrow \so(T_x\OO^\perp)$ whose image is a Lie subalgebra and a $\GG(x)$-submodule where the last structure comes from the homomorphism of Lie algebras $\GG(x) \rightarrow \so(T_x\OO^\perp)$ given by Proposition~\ref{prop:GG(x)andHH0(x)}. By the previous discussion, the latter homomorphism is in fact an isomorphism with $\GG(x) \simeq \so(p,q)$ simple. It follows that $\lambda_x^\perp(\HH_0(x))$ is either $0$ or $\so(T_x\OO^\perp)$. But then, Proposition~\ref{prop:o_O_more_intertwining} implies that $\lambda_x^\perp(\HH_0(x)) = 0$ because $\Omega_x$ is injective. This proves that $\HH_0(x) = 0$.

\begin{lemma}\label{lem:so-and-dim<p+q}
    With the hypotheses and assumptions of this subsection, let $S$ and $U$ be subsets of $\widetilde{M}$ given by Theorem~\ref{thm:rho-map} and Proposition~\ref{prop:HH-transitivity}, respectively. Then, for almost every $x \in S\cap U$ we have
    \begin{enumerate}
      \item $\HH_0(x) = 0$.
      \item $T_x\OO^\perp \simeq \R^{p,q}$ as $\GG(x)$-modules, which is the lowest dimensional non-trivial irreducible module, and admits a unique (up to a constant) $\GG(x)$-invariant non-degenerate symmetric bilinear form.
    \end{enumerate}
    Furthermore, $\HH$ is isomorphic as a Lie algebra to either $\so(p,q+1)$ or $\so(p+1,q)$.
\end{lemma}
\begin{proof}
    The first two claims have already been proved above.

    Next, we observe that by Proposition~\ref{prop:VV(x)} we have for almost every $x \in \widetilde{M}$ a decomposition
    \[
        \HH = \GG(x) \oplus \VV(x),
    \]
    where  $\GG(x)$ is Lie subalgebra of $\HH$ isomorphic to $\so(p,q)$ and $\VV(x)$ is a $\GG(x)$-submodule. In other words, we have $[\GG(x), \VV(x)] \subset \VV(x)$.

    If $\HH$ is not semisimple, then there is a semisimple Lie subalgebra $\mathcal{S}$ of $\HH$ containing $\GG(x)$ such that
    \[
        \HH = \mathcal{S} \oplus rad(\HH).
    \]
    Hence, the first decomposition of $\HH$ obtained above proves that $rad(\HH) = \VV(x)$ and $\mathcal{S} = \GG(x)$. This shows that $\HH$ is isomorphic as a Lie algebra to $\so(p,q) \ltimes \R^{p,q}$.

    Recall that $\HH$ is a Lie subalgebra of the Killing Lie algebra of $\widetilde{M}$. Hence, the previous discussion shows that there is a local action of $\wtSO_0(p,q) \ltimes \R^{p,q}$ in a neighborhood of $x$ in $\widetilde{M}$. Then, the arguments from Subsection~3.1 from \cite{OQ-SO} can be used to conclude that $T\OO^\perp$ is integrable in a neighborhood of $x$. By Proposition~\ref{prop:o_O_homomorphisms} it follows that $\Omega$ vanishes in a neighborhood of $x$ and so, by analyticity, it vanishes in all of $\widetilde{M}$. This implies that $T\OO^\perp$ is everywhere integrable which is a contradiction.

    It follows that $\HH$ is semisimple and the decomposition $\HH = \GG(x) \oplus \VV(x)$ together with the relation $[\GG(x), \VV(x)] \subset \VV(x)$ is easily seen to imply that $\HH$ is simple. From this it is easily seen that $\HH$ has precisely the required isomorphism type. The details can be found in \cite{OQ-SO}.
\end{proof}

\subsection{The group $G = G_{2(2)}$ and $\dim(M) \leq 21$}
Let us now assume that $G$ is the non-compact simply connected Lie group of type $G_2$ over $\R$ and that $\dim(M) \leq 21$. Hence, $T\OO^\perp$ is again a direct summand which we are assuming to be non-integrable. With the same notation as in the previous case, for almost every $x \in S\cap U$ the map $\Omega_x \not= 0$, and $ev_x(\HH) = T_x \widetilde{M}$. We fix $x$ with these properties.

Then, $T_x \OO^\perp$ is an non-trivial $\GG(x)$-module with dimension at most $7$. Since $\GG(x) \simeq \g_{2(2)}$, which is split of type $\g_2$, it follows from Weyl's dimension formula that $T_x \OO^\perp$ is isomorphic to the $7$-dimensional irreducible $\g_{2(2)}$-module, which is the lowest dimensional. We recall that, since $\g_{2(2)}$ is split, $T_x\OO^\perp$ admits a unique (up to a constant) $\GG(x)$-invariant non-degenerate symmetric bilinear form.

Since $\Omega_x : \so(T_x\OO^\perp) \rightarrow \g_{2(2)}$ is a non-zero homomorphism of $\GG(x)$-modules, it follows that $\so(T_x\OO^\perp)$ contains a submodule isomorphic to $\g_{2(2)}$. Let us write
\[
    \so(T_x\OO^\perp) \simeq \g_{2(2)} \oplus W,
\]
for some $\GG(x)$-submodule of $\so(T_x\OO^\perp)$. In particular, $\dim(W) = 7$ and so Proposition~\ref{prop:V0-inTOOperp-dim4} and Weyl's formula again imply that $W$ is isomorphic to the $7$-dimensional $\g_{2(2)}$-module. We recall that $\lambda_x^\perp : \GG(x) \rightarrow \so(T_x\OO^\perp)$ is injective, and so the previous description shows that $\lambda_x^\perp(\GG(x))$ is the unique $\GG(x)$-submodule of $\so(T_x\OO^\perp)$ isomorphic to $\g_{2(2)}$. In other words, we have the decomposition
\[
    \so(T_x\OO^\perp) = \lambda_x^\perp(\GG(x)) \oplus W.
\]

Since the $\GG(x)$-module structure is given by the Lie brackets of $\so(T_x\OO^\perp)$ it follows that the map $\wedge^2 W \rightarrow \so(T_x\OO^\perp)$ defined by the Lie brackets is a homomorphism of $\GG(x)$-modules. Hence, the only possibilities for $[W,W]$ are to be $0$, $W$, $\lambda_x^\perp(\GG(x))$ or $\so(T_x\OO^\perp)$. The first and second imply that $W$ is a proper ideal of $\so(T_x\OO^\perp)$, which is a contradiction. The third one yields a symmetric pair of the form $(\so(3,4),\g_{2(2)})$ which does not exist (see \cite{Berger}). Hence, we have $[W,W] = \so(T_x\OO^\perp)$.

On the other hand, Proposition~\ref{prop:lambdaperp_on_HH0(x)} shows that $\lambda_x^\perp$ realizes $\HH_0(x)$ as a $\GG(x)$-submodule of $\so(T_x\OO^\perp)$ which is also a Lie subalgebra of $\so(T_x\OO^\perp)$. From the previous discussion the only possibilities for the image of $\HH_0(x)$ under $\lambda_x^\perp$ are $0$, $\lambda_x^\perp(\GG(x))$ and $\so(T_x\OO^\perp)$. Hence, Proposition~\ref{prop:o_O_more_intertwining} implies that $\HH_0(x) = 0$ because the kernel of $\Omega_x$ is $W$.

\begin{lemma}\label{lem:g22-and-dim<21}
    With the hypotheses and assumptions of this subsection, let $S$ and $U$ be subsets of $\widetilde{M}$ given by Theorem~\ref{thm:rho-map} and Proposition~\ref{prop:HH-transitivity}, respectively. Then, for almost every $x \in S\cap U$ we have
    \begin{enumerate}
      \item $\HH_0(x) = 0$.
      \item $T_x\OO^\perp$ is a non-trivial, irreducible and $7$-dimensional $\GG(x)$-module, and so it is the lowest dimensional non-trivial irreducible $\g_{2(2)}$-module. Also, $T_x\OO^\perp$ admits a unique (up to a constant) $\GG(x)$-invariant non-degenerate symmetric bilinear form.
    \end{enumerate}
    Furthermore, $\HH$ is isomorphic as a Lie algebra to $\so(3,4)$.
\end{lemma}
\begin{proof}
    Again, it only remains to obtain the isomorphism type of $\HH$. For this we fix $x$ satisfying the first two claims.

    As before we have a decomposition
    \[
        \HH = \GG(x) \oplus \VV(x),
    \]
    with $\GG(x)$ a Lie subalgebra isomorphic to $\g_{2(2)}$ and $\VV(x)$ a $\GG(x)$-submodule isomorphic to the only $7$-dimensional irreducible $\g_{2(2)}$-module.

    We consider a decomposition
    \[
        \HH = \mathcal{S} \oplus rad(\HH),
    \]
    with $\mathcal{S} \supset \GG(x)$. As in the proof of Lemma~\ref{lem:so-and-dim<p+q}, one can prove that assuming that $\HH$ is not semisimple implies that the normal bundle $T\OO^\perp$ is integrable.

    Hence, we conclude that $\HH$ is semisimple, and we already have
    \[
        [\GG(x), \VV(x)] \subset \VV(x).
    \]
    It is easy to see from this (see \cite{QuirogaG2}) that $\HH$ is (real) simple Lie algebra with dimension $21$ which is in fact isomorphic to $\so(3,4)$.
\end{proof}

\subsection{Non-compact real forms of type $F_4$ and $\dim(M) \leq 78$}
Let us now assume that $G$ is a simply connected non-compact Lie group of type $F_4$ over $\R$. Recall that there are exactly two such groups up to isomorphism (see \cite{Helgason}). We will also assume that $\dim(M) \leq 78$. Since $\dim(G) = 52$, we conclude from Proposition~\ref{prop:dimM<2dimG} that $T\OO^\perp$ is direct summand of $T\widetilde{M}$ with rank at most $26$, which we are assuming to be non-integrable.

With the same notation as in the previous subsections, we note that for almost every $x \in S \cap U$ the map $\Omega_x : \so(T_x\OO^\perp) \rightarrow \g$ is surjective. Hence, there is a $\GG(x)$-submodule $W$ of $\so(T_x\OO^\perp)$ so that
\[
    \so(T_x\OO^\perp) \simeq \g \oplus W.
\]
Note that $\dim(T_x\OO^\perp) \leq 26$ implies that $\dim(\so(T_x\OO^\perp)) \leq 325$ and $\dim(W) \leq 273$.

On the other hand, it is known that every irreducible module, over a non-compact real Lie algebra of type $F_4$, remains irreducible after it is complexified (see \cite{Oni}). Hence, one can use (complex) Weyl's dimension formula to compute the dimensions of the irreducible $\g$-modules. From this it follows that, below dimension $273$, there is exactly one irreducible non-trivial $\g$-module for each of the dimensions $26$, $52$ and $273$.

We conclude that $T_x\OO^\perp$ is necessarily the irreducible $\g$-module of dimension $26$. By \cite{Oni}, every irreducible module over a non-compact form of type $F_4$ over $\R$ complexifies to an irreducible module. This implies that $T_x\OO^\perp$ admits a unique (up to a constant) $\GG(x)$-invariant non-degenerate symmetric bilinear form.

Hence, $\so(T_x\OO^\perp)$ and $W$ have dimensions exactly $325$ and $273$. Also, by the above remark on the low dimensional $\g$-modules, if $W$ is not irreducible, then it is the direct sum of $k_0$ trivial modules, $k_1$ irreducible modules with dimension $26$ and $k_2$ irreducible modules with dimension $52$, so that we have
\[
    273 = k_0 + (k_1 + 2k_2)26.
\]
In view of Proposition~\ref{prop:V0-inTOOperp-dim4} we have $k_0 \leq 4$ and so the previous identity of non-negative integers is impossible. Hence, $W$ is the $273$-dimensional irreducible $\GG(x)$-module.

The same arguments used in the previous subsections show that $\lambda_x^\perp(\GG(x))$ is the unique $\GG(x)$-submodule of $\so(T_x\OO^\perp)$ isomorphic to $\g$ and the decomposition
\[
    \so(T_x\OO^\perp) = \lambda_x^\perp(\GG(x)) \oplus W
\]
exhibits the left hand-side as a sum of two irreducible $\GG(x)$-modules, the first one isomorphic to $\GG(x)$ itself and the second one with dimension $273$.

A similar reasoning as above shows that $\HH_0(x)$ is either $0$ or its image under $\lambda_x^\perp$ is all of $\so(T_x\OO^\perp)$. But latter contradicts Proposition~\ref{prop:o_O_more_intertwining} since the kernel of $\Omega_x$ is $W$. And so we conclude that $\HH_0(x) = 0$.

\begin{lemma}\label{lem:typeF-and-dim<78}
    With the hypotheses and assumptions of this subsection, let $S$ and $U$ be subsets of $\widetilde{M}$ given by Theorem~\ref{thm:rho-map} and Proposition~\ref{prop:HH-transitivity}, respectively. Then, for almost every $x \in S\cap U$ we have
    \begin{enumerate}
      \item $\HH_0(x) = 0$.
      \item $T_x\OO^\perp$ is a non-trivial, irreducible and $26$-dimensional $\GG(x)$-module, and so it is the lowest dimensional non-trivial irreducible $\g$-module. Also, $T_x\OO^\perp$ admits a unique (up to a constant) $\GG(x)$-invariant non-degenerate symmetric bilinear form.
    \end{enumerate}
    Furthermore, $\HH$ is isomorphic as a Lie algebra to a non-compact simple real Lie algebra of type $E_6$ over $\R$.
\end{lemma}
\begin{proof}
    Note that we only need to prove the required isomorphism class for $\HH$, for which we fix $x$ so that the first two claims are satisfied.

    By Proposition~\ref{prop:VV(x)} we have a decomposition
    \[
        \HH = \GG(x) \oplus \VV(x),
    \]
    where $\GG(x)$ is a Lie subalgebra of $\HH$ isomorphic to $\g$ (currently a non-compact simple Lie algebra of type $F_4$ over $\R$) and $\VV(x)$ isomorphic to the $26$-dimensional irreducible $\g$-module.

    Then, we can use the arguments from the previous subsections to show that $\HH$ is a simple Lie algebra, and we already know that $\HH$ has dimension $78$.

    First, we observe that there is no simple complex Lie algebra with dimension $39$ (see \cite{Helgason}) and so $\HH$ is a real form of a simple complex Lie algebra with dimension $78$.

    Second, we note that the only simple complex Lie algebras with dimension $78$ are precisely $\so(13,\C)$, $\spi(12,\C)$ and $\mathfrak{e}_6^\C$. Hence, the previous argument shows that $\HH^\C$ is isomorphic to one of these $78$-dimensional complex Lie algebras. Since $\GG(x)$ is a Lie subalgebra of $\HH$ isomorphic to $\g$, the first two possibilities imply the existence of a non-trivial representation $\mathfrak{f}_4^\C \rightarrow \mathfrak{gl}(13,\C)$, which is impossible since the lowest dimension where $\mathfrak{f}_4^\C$ has a non-trivial representation is $26$.

    We conclude that $\HH^\C \simeq \mathfrak{e}_6^\C$ and so $\HH$ is a non-compact simple real Lie algebra of type $E_6$ over $\R$.
\end{proof}

\subsection{A geometric local rigidity theorem}

For the following result an important assumption is that the centralizer $\HH$ has Lie subalgebra $\HH_0(x) = 0$ for suitable points $x$. We also assume that for such points $x$, the $\GG(x)$-module $T_x\OO^\perp$ is irreducible. Under these conditions it is possible to provide a geometric description of the pseudo-Riemannian metric on $M$. Although the required hypotheses might seem restrictive, the previous subsections provide examples where they are satisfied and thus yield local rigidity conclusions from more elementary assumptions.

\begin{theorem}[Local rigidity of $G$-actions]
\label{thm:local-rigidity-general}
    Let $G$ be a connected non-compact simple Lie group and let $M$ be a finite volume analytic pseudo-Riemannian manifold. Suppose that there is an analytic and isometric $G$-action on $M$ with a dense orbit and that $\dim(M) < 2 \dim(G)$. Let $\HH$ be the centralizer in $\Kill(\widetilde{M})$ of the $\widetilde{G}$-action on $\widetilde{M}$ and let $H$ be a connected Lie group whose Lie algebra is $\HH$. With the notation of Theorem~\ref{thm:rho-map} and Proposition~\ref{prop:HH-transitivity}, assume that for almost every $x \in S\cap U$ the following holds
    \begin{enumerate}
      \item $\HH_0(x) = 0$.
      \item $T_x\OO^\perp$ is an irreducible $\GG(x)$-module for which there exists a unique (up to a constant) $\GG(x)$-invariant non-degenerate symmetric bilinear form.
    \end{enumerate}
    Let us also assume that $\HH$ is a simple Lie algebra and that both $\HH$ and $\g$ have simple complexifications.

    Then, the pseudo-Riemannian metric of $M$ can be rescaled on the bundles tangent and normal to the $G$-orbits to obtain a new $G$-invariant pseudo-Riemannian metric $\widehat{g}$ so that $(M,\widehat{g})$ is locally isometric to $H$ endowed with a bi-invariant pseudo-Riemannian metric. Hence, $M$ admits a $G$-invariant locally symmetric pseudo-Riemannian metric locally isometric to a bi-invariant metric on $H$.
\end{theorem}
\begin{proof}
    Choose some $x \in S\cap U$ for which the conditions in the hypotheses hold. Hence, Proposition~\ref{prop:VV(x)} implies that
    \[
        \HH = \GG(x) \oplus \VV(x)
    \]
    with the evaluation map $ev_x$ at $x$ mapping isomorphically $\GG(x)$ and $\VV(x)$ onto $T_x\OO$ and $T_x\OO^\perp$, respectively.

    Let us consider a local action of $\widetilde{H}$ on $\widetilde{M}$ induced from the fact that $\HH$ is a Lie subalgebra of $\Kill(\widetilde{M})$. This local action commutes with the $\widetilde{G}$-action because $\HH$ centralizes the latter. In particular, the local $\widetilde{H}$-action preserves the decomposition $T\widetilde{M} = T\OO \oplus T\OO^\perp$. We will use a right action notation for this local action, which accounts for the fact that $\Kill(\widetilde{M})$ is anti-isomorphic to the (locally defined) Lie algebras of local isometries.

    For $h$ in a neighborhood of the identity in $\widetilde{H}$ we can define an orbit map $\varphi$ given by
    \[
        h \mapsto xh,
    \]
    which is a local diffeomorphism from a neighborhood of the identity onto a neighborhood of $x$. This map is clearly $\widetilde{H}$-equivariant. We also note that $d\varphi_e = ev_x$, the evaluation at $x$ for the Lie algebra $\HH$, and so it is a homomorphism of $\GG(x)$-modules by Proposition~\ref{prop:GG(x)andHH0(x)}. It also follows that
    \[
        d\varphi_e(\GG(x)) = T_x\OO, \quad
        d\varphi_e(\VV(x)) = T_x\OO^\perp.
    \]

    Let us endow $\widetilde{H}$ with the bi-invariant pseudo-Riemannian metric induced from the Killing form of its Lie algebra $\HH$. Since the complexification of $\HH$ is simple, this is (up to a constant) the only bi-invariant metric on $\widetilde{H}$. Since $\dim(\HH) = \dim(M) < 2 \dim(G)$, it follows that $\GG(x)$ is a non-degenerate subspace of $\HH$ for the metric. This implies that such metric at $\GG(x)$ is a constant multiple of the Killing form of $\GG(x)$, which also uses that $\GG(x)$ has simple complexification. Also, since $\VV(x) \simeq T_x\OO^\perp$ as $\GG(x)$-modules, the metric on $\VV(x)$ is unique up to a constant.

    Let us denote with $g$ the $G$-invariant pseudo-Riemannian metric given on $M$ and with the same symbol the corresponding metric on $\widetilde{M}$. Since the representations of $\GG(x)$ on $T_x\OO$ and $T_x\OO^\perp$ are metric preserving for such $g$, it follows from the previous discussion that there exists non-zero constants $c_1$ and $c_2$ such that
    \[
        K = c_1\varphi_e^*(g|_{T_x\OO}) +
            c_2\varphi_e^*(g|_{T_x\OO^\perp}),
    \]
    where $K$ is the Killing form of $\HH$. Hence, if we rescale the metric $g$ on $M$ along the bundles $T\OO$ and $T\OO^\perp$ and define
    \[
        \widehat{g} = c_1 g|_{T_x\OO} +
                        c_2 g|_{T_x\OO^\perp},
    \]
    then $d\varphi_e$ is an isometry for the new corresponding metric on $\widetilde{M}$. Note that the metric $\widehat{g}$ is still $G$-invariant in $M$ because the $G$-action preserves the decomposition $T\widetilde{M} = T\OO \oplus T\OO^\perp$ as well as $g$.

    On the other hand, $\varphi$ is $\widetilde{H}$-equivariant and the $\widetilde{H}$-action preserves the metric on its domain as well as the decomposition $T\widetilde{M} = T\OO \oplus T\OO^\perp$ (because it acts by isometries of $g$ and centralizes the $\widetilde{G}$-action). This implies that $\varphi$ is an isometry from a neighborhood of the identity in $\widetilde{H}$, with its bi-invariant metric, onto a neighborhood of $x$ in $\widetilde{M}$, with the new metric $\widehat{g}$.

    From the above discussion we conclude that $(M,\widehat{g})$ is isometric, in a neighborhood of some point, to the symmetric space $\widetilde{H}$ (with its bi-invariant metric). Hence, the pseudo-Riemannian manifold $(M,\widehat{g})$ satisfies $\nabla R = 0$ on a non-empty open subset. However, it is clear that $\widehat{g}$ is analytic as a consequence of the analyticity of $(M,g)$ and the $G$-action, and the construction above. We conclude that $\nabla R = 0$ (for the metric $\widehat{g}$) on all of $M$ and so that $(M,\widehat{g})$ is a local symmetric space at every point.

    We observe that a connected locally symmetric space is locally homogeneous, i.e.~any two points have isometric neighborhoods. This can be seen by connecting the points with piecewise geodesics and using the reflections on the midpoints of the geodesic segments.

    The above remarks show that $(M,\widehat{g})$ is locally symmetric and isometric to $H$ with a bi-invariant metric.
\end{proof}

An application of Lemmas~\ref{lem:so-and-dim<p+q}, \ref{lem:g22-and-dim<21}, \ref{lem:typeF-and-dim<78} and Theorem~\ref{thm:local-rigidity-general} yield the following consequences. The first two improve results found in \cite{OQ-SO} and \cite{QuirogaG2}, respectively, where completeness of $M$ is also assumed. Here we have dropped the completeness assumption while still providing a local geometric description of the manifold $M$. The third result is completely new to the best of our knowledge.

We observe that for $n = p+q$ we have 
\[
    \frac{n(n+1)}{2} < n(n-1) = 2 \dim (\SO_0(p,q)),
\]
for every $n > 3$.
Hence, the conditions $\dim(M) \leq (p+q)(p+q+1)/2$ and $p+q > 4$ imply in the next result that $\dim(M) < 2 \dim(\SO_0(p,q))$ as required by Theorem~\ref{thm:local-rigidity-general}.

\begin{corollary}\label{cor:so-local-rigidity}
    Let $M$ be a finite volume analytic pseudo-Riemannian manifold with an isometric and analytic $\SO_0(p,q)$-action with a dense orbit, where we assume that $p,q \geq 1$ and $p+q > 4$. Suppose that $\dim(M) \leq (p+q)(p+q+1)/2$ and that the normal bundle to the $\SO_0(p,q)$-orbits is non-integrable. Then, $M$ admits a $\SO_0(p,q)$-invariant pseudo-Riemannian metric for which it is locally isometric to either $\SO_0(p,q+1)$ or $\SO_0(p+1,q)$ endowed with some bi-invariant metric.
\end{corollary}

For the next result we note that $\dim(M) \leq 21$ clearly implies $\dim(M) < 28 = 2\dim(G_{2(2)})$ as required by Theorem~\ref{thm:local-rigidity-general}.

\begin{corollary}\label{cor:g22-local-rigidity}
    Let $M$ be a finite volume analytic pseudo-Riemannian manifold with an isometric and analytic $G_{2(2)}$-action with a dense orbit, where $G_{2(2)}$ is the simply connected non-compact simple real Lie group of type $G_2$. Suppose that $\dim(M) \leq 21$ and that the normal bundle to the $G_{2(2)}$-orbits is non-integrable. Then, $M$ admits a $G_{2(2)}$-invariant pseudo-Riemannian metric for which it is locally isometric to $\SO_0(3,4)$ endowed with some bi-invariant metric.
\end{corollary}

Finally, the assumption $\dim(M) \leq 78$ in the next corollary implies $\dim(M) < 104 = 2\dim(G)$ for $G$ a type $F_4$ Lie group, so again the needed bound on the dimension is satisfied.

\begin{corollary}\label{cor:f4-local-rigidity}
    Let $M$ be a finite volume analytic pseudo-Riemannian manifold with an isometric and analytic $G$-action with a dense orbit, where $G$ is a simply connected non-compact simple real Lie group of type $F_4$. Suppose that $\dim(M) \leq 78$ and that the normal bundle to the $G$-orbits is non-integrable. Then, $M$ admits a $G$-invariant pseudo-Riemannian metric for which it is locally isometric to a non-compact simple real Lie group of type $E_6$ over $\R$ endowed with some bi-invariant metric.
\end{corollary}

\section{Global rigidity for complete analytic pseudo-Riemannian manifolds}

In this section we build from the local description obtained in the previous one using analyticity. For this we will now add one more assumption: completeness of $M$. The following result will be fundamental to obtain our global description. A proof can be found in \cite{OQ-SO}, and we also refer to \cite{ONeill}.

\begin{proposition}\label{prop:completeness-actions}
    Let $N$ be a complete pseudo-Riemannian manifold. Then, the Lie algebra of the isometry group of $N$ is anti-isomorphic to $\Kill(N)$. In particular, for every Lie subalgebra $\h$ of $\Kill(N)$ and $H$ a simply connected Lie group with Lie algebra $\h$ there is an isometric right $H$-action on $N$ so that the map
	\begin{align*}
		\mathrm{Lie}(H) &\rightarrow \h \\
			X &\mapsto X^*
	\end{align*}
    is the identity map. In other words, the local action of $\h$ integrates to an isometric right action of $H$.
\end{proposition}

We now state and prove our main global rigidity result. An important difference with Theorem~\ref{thm:local-rigidity-general}, the local rigidity result, is the extra assumption of completeness of the manifold $M$. This turns out to yield a considerable stronger conclusion: the manifold $M$ is, up to a finite covering, of the form $H/\Gamma$ for some non-compact simple group $H$ and $\Gamma$ a lattice in $H$.

\begin{theorem}[Global rigidity of $G$-actions]
\label{thm:global-rigidity-general}
    Let $G$ be a connected non-compact simple Lie group and let $M$ be a finite volume complete analytic pseudo-Riemannian manifold. Suppose that there is an analytic and isometric $G$-action on $M$ with a dense orbit and that $\dim(M) < 2 \dim(G)$. Let $\HH$ be the centralizer in $\Kill(\widetilde{M})$ of the $\widetilde{G}$-action on $\widetilde{M}$ and let $H$ be a connected Lie group whose Lie algebra is $\HH$. With the notation of Theorem~\ref{thm:rho-map} and Proposition~\ref{prop:HH-transitivity}, assume that for almost every $x \in S\cap U$ the following holds
    \begin{enumerate}
      \item $\HH_0(x) = 0$.
      \item $T_x\OO^\perp$ is an irreducible $\GG(x)$-module for which there exists a unique (up to a constant) $\GG(x)$-invariant non-degenerate symmetric bilinear form.
    \end{enumerate}
    Also assume that both $\HH$ and $\g$ have simple complexifications, that $\HH$ is not of Hermitian type and that the image of every non-trivial homomorphism $\g \rightarrow \HH$ is a maximal subalgebra of $\HH$.

    Then, there exists a finite covering map $\widehat{M} \rightarrow M$,  a lattice $\Gamma \subset \widetilde{H}$, a homomorphism $\varphi : \widetilde{G} \rightarrow \widetilde{H}$, and a diffeomorphism
    \[
        \psi : \widehat{M} \rightarrow \widetilde{H}/\Gamma,
    \]
    which is $\varphi$-equivariant, where the $\widetilde{G}$-action on $\widehat{M}$ is obtained by lifting the $G$-action on $M$.

    Furthermore, one can choose the metric $\widehat{g}$ on $M$ from Theorem~\ref{thm:local-rigidity-general} so that $\psi$ is an isometry where $\widetilde{H}/\Gamma$ carries the metric induced from a bi-invariant metric on $\widetilde{H}$. In particular, the universal covering space $\widetilde{M}$ is a globally symmetric space with respect to the metric $\widehat{g}$.
\end{theorem}
\begin{proof}
    This will build from the proof of Theorem~\ref{thm:local-rigidity-general} and so we will follow its notation. Hence, we start by fixing $x \in S\cap U \subset \widetilde{M}$ for which the hypotheses hold.

    In the first place, the local $\widetilde{H}$-action on $\widetilde{M}$ can be extended to a global action that we consider as an action on the right. We have an orbit map
    \begin{align*}
      \widetilde{F} : \widetilde{H} &\rightarrow \widetilde{M} \\
      h &\mapsto xh,
    \end{align*}
    which is a local diffeomorphism at the identity. Since $\widetilde{F}$ is $\widetilde{H}$-equivariant it follows that it is in fact a local diffeomorphism everywhere.

    As in the proof of Theorem~\ref{thm:local-rigidity-general} we rescale the metric on $M$ to obtain a new metric $\widehat{g}$ so that $\widetilde{F}$ is a local isometry where $\widetilde{H}$ carries the bi-invariant pseudo-Riemannian metric defined by the Killing form of its Lie algebra. Note that the local isometry condition holds everywhere by the $\widetilde{H}$-equivariance of $\widetilde{F}$. Since $\widetilde{H}$ is complete with the chosen metric, it follows that $\widetilde{F}$ is a covering map (see \cite{ONeill}), and so, by the simply connectedness of $\widetilde{M}$, $\widetilde{F}$ is an isometry.

    By the previous discussion, there is a pseudo-Riemannian covering map
    \[
        \pi : \widetilde{H} \rightarrow (M,\widehat{g}),
    \]
    where $\widetilde{H}$ carries (from now on) the bi-invariant metric induced from the Killing form of its Lie algebra.

    We now lift the isometric $G$-action on $M$ to a $\widetilde{G}$-action on $\widetilde{H}$ thus inducing a homomorphism
    \[
        \varphi : \widetilde{G} \rightarrow \Iso_0(\widetilde{H}).
    \]
    We recall that $\Iso_0(\widetilde{H}) = L(\widetilde{H})R(\widetilde{H})$, the left and right translations by $\widetilde{H}$ (see \cite{QuirogaAnnals}). Then, the equivariance of the map $\widetilde{F}$ shows that the right translations define an $\widetilde{H}$-action that integrates the centralizer $\HH$ of the $\widetilde{G}$-action. This implies that $\varphi$ is in fact a homomorphism $\widetilde{G} \rightarrow L(\widetilde{H}) = \widetilde{H}$, which we will consider so from now on.

    On the other hand, through the covering map $\pi$, we can consider $\pi_1(M)$ a subgroup of $\Iso(\widetilde{H})$. Since $\Iso_0(\widetilde{H})$ has finite index in $\Iso(\widetilde{H})$ (see \cite{QuirogaAnnals}), it follows that
    \[
        \Gamma_1 = \pi_1(M) \cap \Iso_0(\widetilde{H})
                = \pi_1(M) \cap L(\widetilde{H})R(\widetilde{H})
    \]
    is a finite index subgroup of $\pi_1(M)$. Since the left $\varphi(\widetilde{G})$-action and the $\pi_1(M)$-action on $\widetilde{M}$ commute it follows that
    \[
        \Gamma_1 \subset L(Z) R(\widetilde{H}),
    \]
    where $Z = Z_{\widetilde{H}}(\varphi(\widetilde{G}))$ is the centralizer of $\varphi(\widetilde{G})$ in $\widetilde{H}$. We will prove that this centralizer $Z$ is finite.

    Let $\mathfrak{z}$ be the Lie algebra of $Z$. Then, $d\varphi(\g) + \mathfrak{z}$ is a Lie subalgebra of $\HH$, and the maximality condition from the hypotheses implies that $\mathfrak{z} \subset d\varphi(\g)$. Hence, we conclude that $\mathfrak{z} = 0$, and so $Z$ is discrete.

    Next, let us consider Cartan involutions $\Theta$ and $\theta$ of $\HH$ and $\g$, respectively, such that $\Theta|_{d\varphi(\g)}$ and $\theta$ are intertwined by $d\varphi$. Let $h \in Z$ be given and write $h = k \exp(X)$, where $k$ lies in the connected Lie subgroup $\widetilde{K}$ of $\widetilde{H}$ with Lie algebra $\mathcal{K}$, the fixed point of set of $\Theta$. Since $\HH$ is not of Hermitian type the subgroup $\widetilde{K}$ is compact. Let $Y \in \g$ be given so that we have
    \[
        d\varphi(Y) = \mathrm{Ad}(h)(d\varphi(Y))
        = \mathrm{Ad}(k\exp(X))(d\varphi(Y)).
    \]
    If $Y$ is an eigenvector for $\theta$, then $d\varphi(Y)$ is an eigenvector for $\Theta$, and so Lemma~1.1.3.7 from \cite{Warner} implies in this case that $[X,d\varphi(Y)] = 0$. But this proves that $X \in \mathfrak{z} = 0$ and so we have $h \in \widetilde{K}$.

    We have proved that $Z$ is a discrete subgroup of the compact group $\widetilde{K}$, thus showing that $Z$ is finite.

    By the previous remarks, it follows that
    \[
        \Gamma = \Gamma_1 \cap R(\widetilde{H})
    \]
    is a finite index subgroup of $\Gamma_1$ and so of $\pi_1(M)$. Also, $\Gamma$ can be considered as a subgroup $\widetilde{H}$. Then, the map $\widetilde{F}$ induces a diffeomorphism
    \[
        F : \widehat{M} = \widetilde{H}/\Gamma \rightarrow \widetilde{M}/\Gamma
    \]
    which is clearly $\widetilde{G}$-equivariant by construction. This in turn yields a finite covering map
    \[
        \psi : \widehat{M} = \widetilde{M}/\Gamma \rightarrow M = \widetilde{M}/\pi_1(M)
    \]
    and this is again $\widetilde{G}$-equivariant. Note that this map is also locally isometric for the metric $\widehat{g}$ on $M$.

    A straightforward computation (see \cite{OQ-SO}) shows that the volume elements on $M$ for the (original) metric $g$ and the metric $\widehat{g}$ are multiples of each other. Hence, the finiteness of the covering map $\psi$ implies that $\widetilde{H}/\Gamma$ has finite volume, thus that $\Gamma$ is a lattice in $\widetilde{H}$.
\end{proof}

Again, we apply Lemmas~\ref{lem:so-and-dim<p+q}, \ref{lem:g22-and-dim<21}, \ref{lem:typeF-and-dim<78} but now with Theorem~\ref{thm:global-rigidity-general} to obtain the following consequences. To complete the proofs of the first two results we only need to observe that $\so(r,s)$ is not of Hermitian type for $r,s \not=2$. These two results can be found in \cite{OQ-SO} and \cite{QuirogaG2}, respectively. The third result is completely new to the best of our knowledge. For this one, we use the proof of Theorem~\ref{thm:global-rigidity-general} for the first claim and the second claim follows directly.

\begin{corollary}\label{cor:so-global-rigidity}
    Let $M$ be a finite volume complete analytic pseudo-Riemannian manifold with an isometric and analytic $\SO_0(p,q)$-action with a dense orbit, where we assume that $p,q > 2$. Suppose that $\dim(M) \leq (p+q)(p+q+1)/2$ and that the normal bundle to the $\SO_0(p,q)$-orbits is non-integrable. Then, there exists a finite covering map $\widehat{M} \rightarrow M$, a lattice $\Gamma$ in $\widetilde{H}$, where $\widetilde{H}$ is isomorphic to either $\wtSO_0(p,q+1)$ or $\wtSO_0(p+1,q)$, a homomorphism $\varphi : \wtSO_0(p,q) \rightarrow \widetilde{H}$ and a diffeomorphism
    \[
        \psi : \widehat{M} \rightarrow \widetilde{H}/\Gamma,
    \]
    which is $\varphi$-equivariant. Furthermore, there is a $\wtSO_0(p,q)$-invariant metric $\widehat{g}$ on $M$ for which $\psi$ is an isometry where $\widetilde{H}/\Gamma$ carries the bi-invariant metric induced from the Killing form of $\widetilde{H}$.
\end{corollary}

\begin{corollary}\label{cor:g22-global-rigidity}
    Let $M$ be a finite volume complete analytic pseudo-Riemannian manifold with an isometric and analytic $G_{2(2)}$-action with a dense orbit. Suppose that $\dim(M) \leq 21$ and that the normal bundle to the $G_{2(2)}$-orbits is non-integrable. Then, there exists a finite covering map $\widehat{M} \rightarrow M$, a lattice $\Gamma$ in $\wtSO_0(3,4)$, a homomorphism $\varphi : G_{2(2)} \rightarrow \wtSO_0(3,4)$ and a diffeomorphism
    \[
        \psi : \widehat{M} \rightarrow \wtSO_0(3,4)/\Gamma,
    \]
    which is $\varphi$-equivariant. Furthermore, there is a $G_{2(2)}$-invariant metric $\widehat{g}$ on $M$ for which $\psi$ is an isometry where $\wtSO_0(3,4)/\Gamma$ carries the bi-invariant metric induced from the Killing form of $\so(3,4)$.
\end{corollary}

\begin{corollary}\label{cor:f4-global-rigidity}
    Let $M$ be a finite volume complete analytic pseudo-Riemannian manifold with an isometric and analytic $G$-action with a dense orbit, where $G$ is a simply connected non-compact Lie group of type $F_4$ over $\R$. Suppose that $\dim(M) \leq 78$ and that the normal bundle to the $G$-orbits is non-integrable. Then, there exist a $G$-invariant pseudo-Riemannian metric $\widehat{g}$ on $M$, a simply connected non-compact Lie group $H$ of type $E_6$ over $\R$, a homomorphism $\varphi : G \rightarrow H$ and a diffeomorphism
    \[
        \widetilde{\psi} : \widetilde{M} \rightarrow H,
    \]
    which is $\varphi$-equivariant and isometric for $\widetilde{M}$ carrying the metric $\widehat{g}$ and $H$ carrying the bi-invariant metric defined by the Killing form of $\HH$.

    Furthermore, if $H$ is not of Hermitian type, then there also exist a lattice $\Gamma$ of $H$ and a finite covering map $\widehat{M} \rightarrow M$ so that $\widetilde{\psi}$ induces a diffeomorphism
    \[
        \psi : \widehat{M} \rightarrow H/\Gamma.
    \]
    In other words, if $H$ is not of Hermitian type, then $(M,\widehat{g})$ is isometric, up to a finite covering, to $H/\Gamma$.
\end{corollary}

\section{Remarks and further results}
In the local and global rigidity theorems one obtains a group $H$ from the action of a group $G$. In some, but not all, of these cases the pair $(H,G)$ is in fact a symmetric pair (see \cite{Berger}). It is also remarkable that some of the representations $\GG(x) \rightarrow \so(T_x\OO^\perp)$ that appear in our arguments yield symmetric pairs as well. It would be interesting to have a better understanding of the role played by symmetric pairs in the rigidity of $G$-actions. A first approach could be to find conditions for the manifold $M$ with its isometric $G$-actions to be of the form $H/\Gamma$ so that $(H,G)$ is a symmetric pair.

In both the local and global theorems we have considered on the group $H$ the bi-invariant metric induced by the Killing form of its Lie algebra. This is the reason to require $\HH$ to have simple complexification as well as the uniqueness condition on the metric in $T_x\OO^\perp$ as stated in Lemmas~\ref{lem:so-and-dim<p+q}, \ref{lem:g22-and-dim<21} and \ref{lem:typeF-and-dim<78}. However, it is possible to drop some of these conditions by allowing $H$ to carry some bi-invariant metric (not necessarily coming from a Killing form).

Our rigidity results are general enough to consider many cases, new and old. However, they do not cover all $G$-actions studied so far. For example, in \cite{OQ-U} we consider the case $G = \mathrm{U}(p,q)$ which does not fit under our current development. However, it should be possible to prove a local rigidity theorem for $\mathrm{U}(p,q)$-actions along the same lines that we have here.

One can also guess about the hypotheses and conclusions in these sort of results. In particular, it is of interest to find if the global rigidity results hold without assuming completeness or find counterexamples. This is for the moment beyond our current techniques.

\end{document}